\documentclass{article}
\usepackage{amsmath}
\usepackage{amssymb}
\usepackage{amsthm}
\usepackage[mathscr]{eucal}
\usepackage[all,knot,cmtip]{xy}
\xyoption{arc}
\usepackage{soul}
\usepackage{graphicx}
\usepackage{color}
\usepackage{cmll}

\numberwithin{equation}{section}
\newtheorem{thm}{Theorem}[section]

\newtheorem{prop}[thm]{Proposition}

\theoremstyle{definition}
\newtheorem{definition}[thm]{Definition}
\newtheorem{rem}[thm]{Remark}

\makeatletter
\newcommand{\prc}{\mathbin{\mathpalette\prc@inner\relax}}
\newcommand{\prc@inner}[2]{%
  \vbox{\offinterlineskip\m@th
    \ialign{%
      ##\cr
      \hidewidth\raisebox{-1.5\height}[0pt][0pt]{$#1.$}\hidewidth\cr
      $#1-$\cr
    }%
  }%
}
\makeatother

\begin{document}
\title{On the sphere spectrum from the viewpoint of linear logic}
\author{Ryo Horiuchi}
\date{}

\maketitle

\begin{abstract}
In this note, we show the category of $\Gamma$-sets with the sphere spectrum, the smash product, and the substitution product gives rise to an isomix linearly distributive category, which is a categorical semantics of the multiplicative fragment of linear logic.
We also show Lydakis' assembly map between these two monoidal products corresponds to the mixor.

\end{abstract}

\section{Introduction}
In \cite{Segal}, Segal introduced $\Gamma$-spaces, which were later shown by Bousfield and Friedlander \cite{BF} to provide a model for connective spectra.
In \cite{Lydakis}, Lydakis constructed a monoidal structure on the category of $\Gamma$-spaces corresponding to the smash product of (connective) spectra.
Based on Lydakis' work, Connes and Consani studied $\Gamma$-sets, the discrete analogue of $\Gamma$-spaces, equipped with the corresponding discrete analogue of the monoidal product, as a framework for algebras over the field with one element in \cite{CC}.

Lydakis also studied another monoidal product on $\Gamma$-spaces, which arises from the monad structure on pointed spaces, and constructed a morphism, called the assembly map, relating these two monoidal products \cite{Lydakis}.
Although the assembly map is not an isomorphism on the nose, it is shown there that the assembly map is an equivalence from the viewpoint of stable homotopy theory.
Moreover, both monoidal products correspond to the smash product of (connective) spectra.
In \cite{CC}, Connes and Consani used the assembly map in the setting of $\Gamma$-sets to investigate the relationship between their approach to the field with one element and Durov's approach \cite{Durov}.

On the other hand, Cockett and Seely introduced linearly distributive categories in \cite{CS}.
They are categories equipped with two monoidal structures satisfying certain compatibility conditions and serve as a categorical semantics for the multiplicative fragment of linear logic proposed by Girard in \cite{Girard}.
Roughly speaking, the two monoidal structures correspond to the {\it times} connective and the {\it par} connective of linear logic respectively, while the compatibility conditions encode the interaction between them.

In this paper, we show that the category of $\Gamma$-sets, equipped with these two monoidal structures and the sphere spectrum as their common unit, admits the structure of a linearly distributive category.
We further show that this linearly distributive category is an isomix linearly distributive category, where the mixor is given by the assembly map.

\section{$\Gamma$-sets and their monoidal products}
In this section, we recall some necessary notions about $\Gamma$-sets from \cite{CC}, \cite{DGM}, and \cite{Lydakis}.
\subsection{$\Gamma \mathchar`-\mathbf{Set}$}

\begin{definition}We let $\Gamma^{op}$ denote the category whose objects are the pointed sets $n_+:=\{0, 1, \cdots, n\}$ for every non-negative integer $n$, where $0$ is the base point, and whose morphisms are pointed maps.

A $\Gamma$-set $X$ is a point preserving functor $X:\Gamma^{op}\to\mathbf{Set}_*$, where $\mathbf{Set}_*$ denotes the category of pointed sets and pointed maps.
We let $\Gamma \mathchar`-\mathbf{Set}$ denote the category of $\Gamma$-sets and natural transformations.
\end{definition}

\begin{definition}For any $n_+\in\Gamma^{op}$, we let $\Gamma^n$ denote the associated $\Gamma$-set given by, for each $k_+\in\Gamma^{op}$, 
\[\Gamma^n(k_+)=\Gamma^{op}(n_+, k_+),\]
the pointed set of morphisms from $n_+$ to $k_+$.
\end{definition}

\begin{definition}The sphere spectrum $\mathbb{S}$ is the inclusion $\Gamma^{op}\to\mathbf{Set}_*$.
\end{definition}

\begin{rem}The $\Gamma$-set $\mathbb{S}$ corresponds to the sphere spectrum in the usual sense \cite[Section 2]{DGM}.

Note that there is an isomorphism $\mathbb{S}\to\Gamma^{1}$ given by
\[k_+\to\Gamma^{op}(1_+, k_+), \quad i\mapsto (1\mapsto i),\]
for each $k_+\in\Gamma^{op}$.
\end{rem}

Let $\mathbb{N}\mathchar`-\mathbf{Mod}$ denote the category of commutative monoids and homomorphisms.
We can associate a $\Gamma$-set to a commutative monoid.

\begin{definition}[\cite{DGM}]Let $A\in\mathbb{N}\mathchar`-\mathbf{Mod}$.
We define $HA\in\Gamma \mathchar`-\mathbf{Set}$ as follows.

For any $n_+\in\Gamma^{op}$, $HA(n_+)=A^{\times n}\in\mathbf{Set}_*$, where the base point is $(0, \cdots, 0)$.
For any $f:m_+\to n_+\in\Gamma^{op}$, 
\[HA(f):A^{\times m}\to A^{\times n}, \quad (a_1, \cdots, a_m)\mapsto (\sum_{j\in f^{-1}(1)}a_j, \cdots, \sum_{j\in f^{-1}(n)}a_j),\]
where the sum over the empty set is understood to be $0$.

This gives rise to a functor $H:\mathbb{N}\mathchar`-\mathbf{Mod}\to\Gamma \mathchar`-\mathbf{Set}$, called the Eilenberg-Mac Lane functor.
\end{definition}

\begin{rem}
For any commutative group $A$, $HA$ corresponds to the Eilenberg-Mac Lane spectrum \cite[Section 2]{DGM}. 

The combinatorics of pointed finite sets recall the additive structure of any $A\in\mathbb{N}\mathchar`-\mathbf{Mod}$ in the following sense.
Consider the map $\alpha:2_+\to 1_+\in\Gamma^{op}$ defined by $\alpha(0)=0$ and $\alpha(1)=\alpha(2)=1$.
This induces the map 
\[HA(\alpha):A\times A\to A, (a_1, a_2)\mapsto a_1+a_2\]
by definition.
By using this, it is shown that $H:\mathbb{N}\mathchar`-\mathbf{Mod}\to\Gamma \mathchar`-\mathbf{Set}$ is fully faithful in \cite{CC}.
\end{rem}

\subsection{Smash and substitution}
We recall two monoidal structures on $\Gamma$-sets, which are similar to the products that play an important role in the study of finite algebraic theories \cite{GF}.
\begin{definition}[\cite{Lydakis, CC}]Let $X, Y\in\Gamma \mathchar`-\mathbf{Set}$. 
The smash product $X\wedge Y\in\Gamma \mathchar`-\mathbf{Set}$ is the left Kan extension of the following diagram:

\[
\xymatrix{
\Gamma^{op}&&\\
\Gamma^{op}\times\Gamma^{op} \ar[u]^{\wedge} \ar[rr]_{X(\mathchar`-)\wedge Y(\mathchar`-)} &&\mathbf{Set}_*,
}\]
where $\wedge:\Gamma^{op}\times\Gamma^{op}\to\Gamma^{op}$ denotes the smash product of pointed finite sets and $X(\mathchar`-)\wedge Y(\mathchar`-)$ denotes the point-wise smash product of pointed sets by abusing notation. 

\end{definition}

\begin{rem}The smash product $m_{+}\wedge n_{+}$ of pointed finite sets is, by definition, isomorphic to $(mn)_{+}$.
Thus, it would be natural to regard the smash product of $\Gamma$-sets as a generalization of the multiplication of natural numbers.

There is an isomorphism $\Gamma^{mn}\to\Gamma^{m}\wedge\Gamma^{n}$ given by
\[\Gamma^{mn}(k_+)\to(\Gamma^{m}\wedge\Gamma^{n})(k_+), \quad (f:m_{+}\wedge n_{+}\to k_+)\mapsto [f, \operatorname{id}_{m_+}\wedge\operatorname{id}_{n_+}].\]
Its inverse is given by $[g, \phi\wedge\varphi]\mapsto g\circ(\phi\wedge\varphi)$.
In particular, we have $\Gamma^{n}\cong\Gamma^{n}\wedge\Gamma^{1}$.

Since each $X\in\Gamma \mathchar`-\mathbf{Set}$ is canonically isomorphic to a colimit of $\Gamma^n$'s, and the smash product commutes with colimits, we have $X\cong X\wedge\mathbb{S}$.
See \cite{Lydakis} for more detail.
\end{rem}

\begin{prop}[\cite{CC}]
 $(\Gamma \mathchar`-\mathbf{Set}, \wedge, \mathbb{S})$ gives rise to a symmetric monoidal category and the Eilenberg-Mac Lane functor 
\[H:(\mathbb{N}\mathchar`-\mathbf{Mod}, \otimes_{\mathbb{N}}, \mathbb{N})\to(\Gamma \mathchar`-\mathbf{Set}, \wedge, \mathbb{S})\]
is a lax monoidal functor.
Moreover, the induced functor on the categories of monoid objects is fully faithful.
\end{prop}

\begin{rem}In this sense, the smash product of $\Gamma$-sets is a generalization of the tensor product of commutative monoids.
From this perspective, \cite{CC} studies monoid objects in the monoidal category $(\Gamma \mathchar`-\mathbf{Set}, \wedge, \mathbb{S})$ as a generalization of rigs.
\end{rem}

\begin{definition}Let $X, Y\in\Gamma \mathchar`-\mathbf{Set}$. 
The substitution product $X\circ Y\in\Gamma \mathchar`-\mathbf{Set}$ is the composition $\tilde{X}\circ Y\in\Gamma \mathchar`-\mathbf{Set}$, where $\tilde{X}:\mathbf{Set}_*\to\mathbf{Set}_*$ is the left Kan extension of the following diagram:

\[
\xymatrix{
\mathbf{Set}_*&&\\
\Gamma^{op} \ar@{^{(}->}[u]^{\mathbb{S}} \ar[rr]_{X} &&\mathbf{Set}_*.
}\]

\end{definition}

\begin{rem}By the universality of the left Kan extension, $\tilde{\mathbb{S}}$ is canonically isomorphic to the identity functor $\operatorname{Id}_{\mathbf{Set}_*}$, which is trivially a monad.

In the same way as for the smash product of $\Gamma$-sets, we define 
\[\tilde{X}\wedge\tilde{Y}:\mathbf{Set}_*\to\mathbf{Set}_*.\]
Note that $\tilde{X}\wedge\tilde{Y}$ is canonically isomorphic to $\widetilde{X\wedge Y}$.
\end{rem}

\begin{prop}[\cite{Lydakis, CC}] $(\Gamma \mathchar`-\mathbf{Set}, \circ, \mathbb{S})$ gives rise to a monoidal category.
\end{prop}
Thus, the monoid objects in this monoidal category can be regarded as a certain class of monads on $\mathbf{Set}_*$.

\begin{definition}[\cite{CC}] Let $S$ be a rig. We associate a functor
\[\Sigma_{S}:\mathbf{Set}_*\to\mathbf{Set}_*\]
to it as follows.

For any $A\in\mathbf{Set}_*$, $\Sigma_{S}(A)$ is the set of finite formal linear combinations $\sum s_a a$ where $a\in A$ and $s_a\in S$ with $0a=0=s*$, where $*\in A$ is the base point.
One can define the addition on $\Sigma_{S}(A)$ induced by $s_aa+s'_aa=(s_a+s'_a)a$ and the base point to be $0$.

For any pointed map $f:A\to B$, we define \[\Sigma_{S}(f):\Sigma_{S}(A)\to\Sigma_{S}(B), \quad \Sigma_{S}(f)(\sum s_a a)=\sum s_a f(a).\]

For each $A\in\mathbf{Set}_*$, we define pointed maps
\[
\mu_A\colon
\Sigma_S(\Sigma_S(A))
\longrightarrow
\Sigma_S(A), \quad
\varepsilon_A\colon
A\longrightarrow
\Sigma_S(A)
\]
by
\[
\mu_A\left(
\sum_i s_i
\left(\sum_j s_{ij}a_{ij}\right)
\right)
=
\sum_{i,j}s_i s_{ij}a_{ij}, \quad
\varepsilon_A(a)=1a,
\]
respectively.
\end{definition}

\begin{rem} 
These give rise to a monad $(\Sigma_S, \mu, \varepsilon)$ on $\mathbf{Set}_*$, and by construction the functor $\Sigma$ from the category of rigs to that of monads on $\mathbf{Set}_*$ is faithful. 
In this sense, the product structures of monads are a generalization of those of rigs.
For a further development of this perspective, see \cite{Durov}, which studies a certain class of monads on $\mathbf{Set}$ as generalized rings.

Note that $\widetilde{HS}$ is isomorphic to $\Sigma_{S}$ for any rig $S$ as an endofunctor on $\mathbf{Set}_*$.
\end{rem}

\subsection{Lydakis' assembly map}
We define Lydakis' assembly map, which is a natural transformation connecting the two monoidal products, $\wedge$ and $\circ$, of $\Gamma$-sets defined above.
In other words, Lydakis' assembly map links these two generalizations of the multiplication of natural numbers.

Let $X, Y$ be $\Gamma$-sets.
We construct a map $X\wedge Y\to X\circ Y$ as follows \cite{Lydakis}.

First we consider the pointed map below for any pointed sets $A, B$
\[\tilde{X}(A)\wedge B\to \tilde{X}(A\wedge B), \alpha\wedge b\mapsto \tilde{X}(\delta_b)(\alpha),\]
where $\delta_b$ is the pointed map $A\to A\wedge B$ given by $\delta_b(a)=a\wedge b$.

In the same manner we have $A\wedge\tilde{Y}(B)\to \tilde{Y}(A\wedge B)$.

By composing the first map with the image of the second map under $\tilde X$, we have
\[\tilde{X}(A)\wedge \tilde{Y}(B)\to \tilde{X}(A\wedge \tilde{Y}(B))\to \tilde{X}(\tilde{Y}(A\wedge B)).\]

Therefore for any map $k_+\wedge l_+\to n_+$ in $\Gamma^{op}$, we have
\[X(k_+)\wedge Y(l_+)\to (X\circ Y)(k_+\wedge l_+)\to (X\circ Y)(n_+).\]

By the universal property of Day convolution, we have the map
\[\operatorname{L}_{X, Y}:X\wedge Y\to X\circ Y,\]
which we call Lydakis' assembly map (for $X$ and $Y$).
Note that the smash product is symmetric but the substitution product is not.

\begin{prop}[\cite{CC}]The identity functor on $\Gamma \mathchar`-\mathbf{Set}$ with Lydakis' assembly maps gives a lax monoidal functor $(\Gamma \mathchar`-\mathbf{Set}, \circ, \mathbb{S}) \to (\Gamma \mathchar`-\mathbf{Set}, \wedge, \mathbb{S})$.
Therefore, via Lydakis' assembly maps, any monoid object in $(\Gamma \mathchar`-\mathbf{Set}, \circ, \mathbb{S})$ can be viewed as a monoid object in $(\Gamma \mathchar`-\mathbf{Set}, \wedge, \mathbb{S})$ as well.
\end{prop}

Connes and Consani studied the assembly maps for $\Gamma$-sets of Eilenberg-Mac Lane type in detail.
\begin{prop}[\cite{CC}]For any rig $S$, the assembly map $\operatorname{L}_{HS, HS}$ is surjective.
Moreover, the assembly map $\operatorname{L}_{H\mathbb{B}, H\mathbb{B}}$ is not injective, where $\mathbb{B}$ denotes the Boolean semifield.
\end{prop}

\begin{rem}Lydakis studied the constructions above in the category of {\it $\Gamma$-spaces} instead of $\Gamma$-sets.
In \cite[Proposition 5.23]{Lydakis}, it is shown that the assembly maps $F\wedge F'\to F\circ F'$ of $\Gamma$-spaces are equivalences from the viewpoint of stable homotopy theory.

In \cite{Masuda}, Masuda extended the smash product of spectra to categorical spectra. 
The resulting smash product is non-symmetric.
Since the smash product $\wedge$ of $\Gamma$-sets is symmetric, whereas the substitution product $\circ$ is non-symmetric (on the nose), it might therefore be natural to expect the latter to correspond to the smash product of (connective) categorical spectra. 
Accordingly, the smash product of $\Gamma$-sets would instead give rise to a different product structure on categorical spectra.
\end{rem}

\section{Linearly distributive category}

In this section, we recall some necessary notions about linearly distributive categories mainly from \cite{SS}.
\subsection{Linearly distributive category}

\begin{definition}[\cite{CS, SS}]
A linearly distributive category is a category $\mathcal{L}$, equipped with two monoidal structures $(\mathcal{L}, \otimes, \top)$ and $(\mathcal{L}, \triangleleft, \bot)$ and two natural transformations given by
\[\partial_{X, Y, Z}^{L}:X\otimes(Y\triangleleft Z)\to (X\otimes Y)\triangleleft Z\]
\[\partial_{X, Y, Z}^{R}:(X\triangleleft Y)\otimes Z\to X \triangleleft (Y\otimes Z)\]
called the left and right linear distributors, respectively, which satisfy appropriate coherence conditions.
\end{definition}

\begin{rem}
Linearly distributive categories were introduced as categorical semantics for the multiplicative fragment of linear logic in \cite{CS}. 
The monoidal products $\otimes$ and $\triangleleft$ correspond respectively to the multiplicative conjunction and multiplicative disjunction.
We do not assume the two monoidal structures are symmetric.
The product $\triangleleft$ may be written as $\parr$ or $\oplus$ in the literature.
We will not pursue the logical aspects further here.
\end{rem}

\begin{definition}[\cite{SS}]A mix linearly distributive category $(\mathcal{L}, \otimes, \triangleleft, \top, \bot, \operatorname{m})$ is a linearly distributive category $(\mathcal{L}, \otimes, \triangleleft, \top, \bot)$ with a morphism $\operatorname{m}:\bot\to\top$, called the mix map, such that for any $X, Y\in\mathcal{L}$ the diagram below commutes

\[
\xymatrix{
  X\otimes Y \ar[rr]^{\operatorname{id}\otimes (\operatorname{u}_{\triangleleft}^{L})^{-1}}\ar[d]_{(\operatorname{u}_{\triangleleft}^{R})^{-1}\otimes\operatorname{id}}&&X\otimes(\bot\triangleleft Y)\ar[rr]^{\operatorname{id}\otimes (\operatorname{m}\triangleleft\operatorname{id})}&&X\otimes(\top\triangleleft Y) \ar[d]^{\partial^{L}}\\
  (X\triangleleft\bot)\otimes Y \ar[d]_{\partial^{R}}&&&& (X\otimes\top)\triangleleft Y\ar[d]^{(\operatorname{u}_{\otimes}^{R})\triangleleft\operatorname{id}}\\
   X\triangleleft(\bot\otimes Y)\ar[rr]_{\operatorname{id}\triangleleft (\operatorname{m}\otimes\operatorname{id})}&& X\triangleleft(\top\otimes Y)\ar[rr]_{\operatorname{id}\triangleleft (\operatorname{u}_{\otimes}^{L})}& &X\triangleleft Y,
}\]
where $\operatorname{u}$'s are the unit isomorphisms.
The composite map $X\otimes Y\to X\triangleleft Y$ is called the mixor.

An isomix linearly distributive category is a mix linearly distributive category whose mix map is an isomorphism.

\end{definition}

\begin{rem}Every monoidal category can be viewed as an isomix linearly distributive category such that the two monoidal structures are the same, the linear distributors are the associators, and the mix map is the identity.
\end{rem}

\subsection{Normal duoidal category}
\begin{definition}[\cite{SS}]A duoidal category is a category $\mathcal{D}$ equipped with two monoidal structures $(\mathcal{D}, \otimes, \top)$, $(\mathcal{D}, \triangleleft, \bot)$, a natural transformation given by
\[\operatorname{int}_{X_1, X_2, X_3, X_4}:(X_1\triangleleft X_2)\otimes(X_3\triangleleft X_4)\to (X_1\otimes X_3)\triangleleft (X_2\otimes X_4),\]
called the interchange law, and two morphisms
\[\top\to\top\triangleleft\top, \quad \bot\otimes\bot\to\bot\]
such that $\triangleleft:\mathcal{D}\times\mathcal{D}\to\mathcal{D}$ and $\bot:*\to\mathcal{D}$ are $\otimes$-lax monoidal functors, where $*$ denotes the trivial category, and the associativity and unitor isomorphisms of $(\mathcal{D}, \triangleleft, \bot)$ are $\otimes$-monoidal.

A duoidal category is called a normal duoidal category if the map $\operatorname{int}_{\top, \bot, \bot, \top}$ is an isomorphism.

\end{definition}

\begin{rem}In any duoidal category $(\mathcal{D}, \otimes, \top, \triangleleft, \bot)$, there is a morphism
\[k:\top\cong\top\otimes\top\cong(\top\triangleleft\bot)\otimes(\bot\triangleleft\top)\xrightarrow{\operatorname{int}_{\top, \bot, \bot, \top}}(\top\otimes\bot)\triangleleft(\bot\otimes\top)\cong\bot\triangleleft\bot\cong\bot,\]
where the isomorphisms are induced by the unit isomorphisms.
In any normal duoidal category, $k$ is an isomorphism.
\end{rem}

\begin{thm}[\cite{SS, GF}]\label{thm:duoidal-isomix}
Every normal duoidal category gives rise to an isomix linearly distributive category.
\end{thm}

\begin{rem}
In \cite{SS}, Spivak and Srinivasan prove more than the theorem above.
They define morphisms of normal duoidal categories and linearly distributive categories, thereby obtaining the categories of normal duoidal categories and linearly distributive categories, respectively.
They then construct a faithful functor from the former to the latter.
\end{rem}

\begin{rem}For any normal duoidal category $(\mathcal{D}, \otimes, \top, \triangleleft, \bot)$, the associated mix map $\operatorname{m}$ is defined to be
\[\operatorname{m}:\bot\cong\bot\triangleleft\bot\cong(\top\otimes\bot)\triangleleft(\bot\otimes\top)\xrightarrow{\operatorname{int}^{-1}_{\top, \bot, \bot, \top}}(\top\triangleleft\bot)\otimes(\bot\triangleleft\top)\cong\top\otimes\top\cong\top,\]
where the unlabeled isomorphisms are induced by the unit isomorphisms.
Furthermore, the associated left distributor $\partial_{X, Y, Z}^{L}$ and the associated right distributor $\partial_{X, Y, Z}^{R}$ are defined as follows:

\[
\begin{aligned}
X\otimes(Y\triangleleft Z)
&\cong (X\triangleleft\bot)\otimes(Y\triangleleft Z) \\
&\xrightarrow{\operatorname{int}_{X,\bot,Y,Z}}
(X\otimes Y)\triangleleft(\bot\otimes Z) \\
&\xrightarrow{(\operatorname{id}_{X}\otimes\operatorname{id}_{Y})
\triangleleft(k^{-1}\otimes\operatorname{id}_{Z})}
(X\otimes Y)\triangleleft(\top\otimes Z) \\
&\cong (X\otimes Y)\triangleleft Z,
\end{aligned}
\]
\[
\begin{aligned}
(X\triangleleft Y)\otimes Z
&\cong (X\triangleleft Y)\otimes(\bot\triangleleft Z) \\
&\xrightarrow{\operatorname{int}_{X, Y, \bot, Z}}
(X\otimes \bot)\triangleleft(Y\otimes Z) \\
&\xrightarrow{(\operatorname{id}_{X}\otimes k^{-1})
\triangleleft(\operatorname{id}_{Y}\otimes\operatorname{id}_{Z})}
(X\otimes \top)\triangleleft(Y\otimes Z) \\
&\cong X\triangleleft (Y\otimes Z),
\end{aligned}
\]
where the unlabeled isomorphisms are induced by the unit isomorphisms.
\end{rem}

\section{Main theorem}
We have the following, which is essentially the same as \cite[Proposition 36]{GF}.
\begin{prop}$(\Gamma \mathchar`-\mathbf{Set}, \wedge, \circ,  \mathbb{S})$ gives rise to a normal duoidal category.
\end{prop}
\begin{proof}It is enough to construct the interchange law 
\[(X_1\circ X_2)\wedge(X_3\circ X_4)\to (X_1\wedge X_3)\circ (X_2\wedge X_4).\]
For any $k_+, l_+\in\Gamma^{op}$, we have the canonical map 
\[\tilde{X_1}(X_2(k_+))\wedge\tilde{X_3}(X_4(l_+))\to (\tilde{X_1}\wedge\tilde{X_3})(X_2(k_+)\wedge X_4(l_+))\]
by the universal property of the left Kan extension.
Note that $\tilde{X_1}\wedge\tilde{X_3}$ is canonically isomorphic to $\widetilde{X_1\wedge X_3}$.
Again by the universal property we have 
\[(\tilde{X_1}\wedge\tilde{X_3})(X_2(k_+)\wedge X_4(l_+))\to(\tilde{X_1}\wedge\tilde{X_3})((X_2\wedge X_4)(k_+\wedge l_+)).\]

The composite of these maps, together with the universal property of the smash product, gives the interchange law.
\end{proof}

\begin{rem}Note that in the normal duoidal category $(\Gamma \mathchar`-\mathbf{Set}, \wedge, \circ,  \mathbb{S})$, the interchange law $\operatorname{int}_{X, \mathbb{S}, \mathbb{S}, Y}$ is isomorphic to the assembly map $\operatorname{L}_{X, Y}$.

Moreover, in the associated isomix linearly distributive category, the linear distributors $\partial^{L}_{X, \mathbb{S}, Y}$ and $\partial^{R}_{X, \mathbb{S}, Y}$ are isomorphic to the interchange law $\operatorname{int}_{X, \mathbb{S}, \mathbb{S}, Y}$.
\end{rem}

By Theorem \ref{thm:duoidal-isomix}, we have the following.

\begin{thm}The quadruple $(\Gamma \mathchar`-\mathbf{Set}, \wedge, \circ,  \mathbb{S})$ gives rise to an isomix linearly distributive category, where the mixors are isomorphic to Lydakis' assembly maps. 
\end{thm}

In this semantics, the multiplicative conjunction is represented by the smash product $\wedge$, the multiplicative disjunction is represented by the substitution product $\circ$, and both $\bot$ and $\top$ are represented by the sphere spectrum $\mathbb{S}$.

The results established above should also hold for $\Gamma$-spaces.
Thus, they would provide a geometric categorical semantics of multiplicative linear logic.
It may be worth noting that, from the viewpoint of stable homotopy theory, the assembly maps become equivalences.


\begin{thebibliography}{99}


\bibitem{BF}
Bousfield, A. K.; Friedlander, E. M.
Homotopy theory of $\Gamma$-spaces, spectra, and bisimplicial sets. Geom. Appl. Homotopy Theory, II, Proc. Conf., Evanston 1977, Lect. Notes Math. 658, 80--130 (1978).

\bibitem{CS}
Cockett, J. R. B.; Seely, R. A. G.
Weakly distributive categories. J. Pure Appl. Algebra 114, No. 2, 133--173 (1997).


\bibitem{CC}
Connes, Alain; Consani, Caterina.
Absolute algebra and Segal's $\Gamma$-rings: au dessous de $\overline{\mathrm{Spec}(\mathbb{Z})}$. J. Number Theory 162, 518--551 (2016).

\bibitem{DGM}
Dundas, Bj{\o}rn Ian; Goodwillie, Thomas G.; McCarthy, Randy.
The local structure of algebraic $K$-theory.
Algebra and Applications 18. London: Springer (ISBN 978-1-4471-4392-5/hbk; 978-1-4471-4393-2/ebook). xv, 435 p. (2013).

\bibitem{Durov}
Durov, Nikolai.
New approach to Arakelov geometry. 
PhD thesis, Univ. Bonn, 2007.

\bibitem{GF}
Garner, Richard; L{\'o}pez Franco, Ignacio.
Commutativity. J. Pure Appl. Algebra 220, No. 5, 1707--1751 (2016).

\bibitem{Girard}
Girard, Jean-Yves.
Linear logic. Theor. Comput. Sci. 50, 1--102 (1987).

\bibitem{Lydakis}
Lydakis, Manos.
Smash products and $\Gamma$-spaces. Math. Proc. Camb. Philos. Soc. 126, No. 2, 311--328 (1999).


\bibitem{Masuda} Masuda, Naruki.
The algebra of categorical spectra, 
PhD thesis, Johns Hopkins University, 2024.



\bibitem{Segal}
Segal, Graeme.
Categories and cohomology theories. Topology 13, 293--312 (1974).

\bibitem{SS}
Spivak, David I.; Srinivasan, Priyaa Varshinee.
What kind of linearly distributive category do polynomial functors form? Theory Appl. Categ. 45, 1748--1781 (2026).




\end{thebibliography}
\end{document}